\documentclass[11pt]{article}
\usepackage{amsmath, amssymb, amsthm, mathrsfs, bm, geometry, mathtools, hyperref, booktabs, xcolor, soul}
\usepackage{appendix}
\usepackage{natbib}
\newtheorem{theorem}{Theorem}[section]
\newtheorem{lemma}[theorem]{Lemma}

\newtheorem{corollary}[theorem]{Corollary}
\newtheorem{definition}[theorem]{Definition}
\newtheorem{assumption}[theorem]{Assumption}
\newtheorem{remark}[theorem]{Remark}

\newcommand{\R}{\mathbb{R}}
\newcommand{\E}{\mathbb{E}}
\newcommand{\Var}{\mathrm{Var}}
\newcommand{\Cov}{\mathrm{Cov}}

\newcommand{\norm}[1]{\left\| #1 \right\|}
\newcommand{\abs}[1]{\left| #1 \right|}

\newcommand{\dd}{\mathrm{d}}

\title{\bf{Mean-Square Convergence of a New Parameterized Leapfrog Scheme for Hamiltonian Systems Driven by Gaussian Process Potentials}}
\author{Sourabh Bhattacharya \\ Indian Statistical Institute} 
\date{}

\begin{document}

\maketitle

\begin{abstract}
This paper establishes the mean-square convergence of a new stochastic, parameterized leapfrog scheme introduced in our companion paper \cite{main} 
for Hamiltonian systems with Gaussian process potentials. We consider a one-step numerical integrator and provide a complete, rigorous analysis under minimal regularity 
assumptions on the Gaussian potential. The key technical contribution is identifying and exploiting the symplectic structure ingrained in our stochastic, parameterized leapfrog 
method. Combined with local truncation error analysis, this leads to a global error bound of $\mathcal{O}(\delta t)$ in mean-square sense. Our results 
establish that although the spatio-temporal model of \cite{main} arises as the anticipated new stochastic leapfrog solution of a system of modified (parameterized) 
stochastic Hamiltonian equations, the new stochastic leapfrog actually solves the traditional stochastic Hamiltonian equations, driven by Gaussian process potential.
\end{abstract}

\tableofcontents

\section{Introduction}
\label{sec:introduction}

Hamiltonian systems with random potentials arise in numerous scientific applications, including Bayesian statistical sampling \cite{girolami2011riemann}, 
molecular dynamics with uncertain parameters \cite{leimkuhler2015efficient}, and stochastic geometric mechanics \cite{bloch2014stochastic}. 
The numerical integration of such systems requires careful analysis due to the interplay between discretization errors and the randomness of the potential.

We consider Hamiltonian systems of the form:

\begin{align}
    \frac{dy}{dt} &= M^{-1} x, \label{eq:original_hamiltonian_y_intro} \\
    \frac{dx}{dt} &= -\nabla V(y; \omega), \label{eq:original_hamiltonian_x_intro}
\end{align}

where \( y, x \in \mathbb{R}^d \) are the position and momentum variables, \( M \) is a symmetric positive definite mass matrix, and 
\( V : \mathbb{R}^d \times \Omega \to \mathbb{R} \) is a \emph{Gaussian process potential} defined on a probability space 
\( (\Omega, \mathscr{F}, P) \). The gradient \( \nabla V \) is a random vector field, making 
the system a \emph{stochastic ordinary differential equation} rather than a stochastic differential equation driven by Wiener processes.

A particularly important class of numerical methods for Hamiltonian systems are symplectic integrators, which preserve the geometric structure of the exact flow 
\cite{hairer2006geometric}. In our companion paper \cite{main}, we introduced a \emph{parameterized stochastic leapfrog scheme} specifically designed for Hamiltonian systems 
with Gaussian process potentials. This scheme extends the classical Störmer--Verlet method with additional parameters that allow for better control of numerical 
errors in the random potential setting. Indeed, \cite{main} proposed the stochastic leapfrog scheme as an anticpated solution to a new parameterized
Hamiltonian system, driven by Gaussian process potential, for their purpose of creating a novel spatio-temporal process with desitable properties. But interestingly, as we
show in this article, the stochastic leapfrog scheme actually solves the traditional system of Hamiltonian -- albeit with Gaussian process potentials, in appropriate 
mean-square sense. We deem this more favorable compared to our orignal intent, since this shows that our stochastic leapfrog, shown to have desirable spatio-temporal properties
in \cite{main}, is also akin to the traditional Hamiltonian that so appropriately explains general phase-space dynamics, with its efficacy enhanced by Gaussian process induced
potential, inculcating in it the extremely flexible nonparametric nature.

When the potential $V$ is a Gaussian process, additional challenges arise: the potential is only almost surely differentiable, and derivatives are correlated random fields. Previous analyses of numerical methods for stochastic Hamiltonian systems have typically focused on either 
deterministic potentials with stochastic perturbations (e.g., \citet{stuart2010deterministic, abdulle2011stiff, leimkuhler2015efficient}), or
completely stochastic differential equations driven by Wiener processes (e.g., \citet{rumelin1983numerical, talay1990second, milstein1995numerical, kloeden1992numerical, hairer2002geometric})..
Our setting differs fundamentally: the potential itself is a Gaussian random field, making the system a random ordinary differential equation rather than 
a stochastic differential equation. And, most importantly, this renders the Hamiltonin system as well as our leapfrog scheme completely nonparametric, which was one of the
main goals of \cite{main}, with respect to their novel spatio-temporal model.

The main contribution of this paper is a complete, rigorous convergence analysis of our new stochastic leapfrog scheme for Hamiltonian systems with 
Gaussian process potentials. 
We address several interconnected technical challenges inherent to the random potential setting. 
First, we establish moment bounds for Gaussian process derivatives under minimal regularity assumptions, 
using Gaussian process theory and the Borell--TIS inequality to obtain integrable bounds on suprema 
of derivatives over compact sets; these bounds are essential for controlling the random terms throughout 
the analysis. Second, we perform a pathwise analysis and construct a modified random ordinary differential 
equation that the scheme follows up to a prescribed order; this requires careful handling of the random 
gradient and Hessian terms while preserving the Hamiltonian structure for almost every realization of the 
potential. Third, we analyze the local truncation error in the mean-square sense by comparing the one-step 
update of our numerical scheme with the exact flow of the original Hamiltonian system, establishing that 
the local error is of order $\mathcal{O}(\delta t^2)$ while accounting for the randomness in the potential. 
Finally, we combine the local truncation error with the stability bound for the traditional leapfrog scheme inherent in our modified leapfrog scheme 
to obtain a global mean-square error of order $\mathcal{O}(\delta t)$ over finite time intervals, thereby 
establishing the convergence of the modified scheme.

%


In contrast with existing analyses of standard symplectic methods, our work specifically addresses the new leapfrog scheme proposed in \cite{main}, which includes additional parameters $\alpha$ and $\beta$ that provide enhanced control over numerical errors in the presence of Gaussian process potentials.

The paper is organized as follows: Section~\ref{sec:preliminaries} introduces the original Hamiltonian system with Gaussian process potential and presents our modified numerical scheme from \cite{main} along with necessary assumptions. Section~\ref{sec:pathwise} provides pathwise analysis and derives the modified equation. Section~\ref{sec:truncation} analyzes the local truncation error. 
Section~\ref{sec:main_convergence} combines these results to prove mean-square convergence. Section~\ref{sec:conclusion} concludes with discussion and future directions.
The proofs of our results are presented in the Appendix.

\section{Preliminaries and problem formulation}
\label{sec:preliminaries}

\subsection{Original Hamiltonian system with Gaussian process potential}

We consider the following Hamiltonian system defined on $\R^d \times \R^d$:

\begin{align}
    \frac{dy}{dt} &= M^{-1} x, \label{eq:original_hamiltonian_y} \\
    \frac{dx}{dt} &= -\nabla V(y; \omega), \label{eq:original_hamiltonian_x}
\end{align}

where:
\begin{itemize}
    \item $y(t) \in \R^d$ is the position vector,
    \item $x(t) \in \R^d$ is the momentum vector,
    \item $M \in \R^{d \times d}$ is a symmetric positive definite mass matrix,
    \item $V: \R^d \times \Omega \to \R$ is a Gaussian process potential defined on a probability space $(\Omega, \mathscr{F}, P)$,
    \item $\nabla V: \R^d \times \Omega \to \R^d$ denotes the gradient (a random vector field).
\end{itemize}

The system \eqref{eq:original_hamiltonian_y}--\eqref{eq:original_hamiltonian_x} is a \emph{stochastic ordinary differential equation} because the potential $V$ is a random field rather than a deterministic function. The Hamiltonian corresponding to this system is:

\[
H(y, x; \omega) = \frac{1}{2} x^T M^{-1} x + V(y; \omega),
\]

which is conserved along solutions when $V$ is smooth in $y$.

\subsection{The stoachstic parameterized leapfrog scheme from \cite{main}}

We consider the following one-step leapfrog scheme introduced in \cite{main}, which, as we show, approximates solutions of 
\eqref{eq:original_hamiltonian_y}--\eqref{eq:original_hamiltonian_x}, defined for integers $n \geq 0$:

\begin{align}
    y_{n+1} &= \beta y_n + \delta t M^{-1}\left(\alpha x_n - \frac{\delta t}{2} \nabla V(y_n; \omega)\right), \label{eq:scheme_y} \\
    x_{n+1} &= \alpha^2 x_n - \frac{\delta t}{2}\left(\alpha \nabla V(y_n; \omega) + \nabla V(y_{n+1}; \omega)\right). \label{eq:scheme_x}
\end{align}

Note that this scheme reduces to the standard leapfrog method when $\alpha = \beta = 1$. The parameters $\alpha$ and $\beta$ provide additional degrees of freedom that can be tuned to optimize performance for Gaussian process potentials, as demonstrated in \cite{main}.

\subsection{Notation and basic definitions}

\begin{itemize}
    \item $y_n, x_n \in \R^d$ are numerical approximations to $y(t_n)$ and $x(t_n)$ at time $t_n = n \delta t$.
    \item $\delta t > 0$ is the time step (discretization parameter).
    \item $\alpha, \beta \in \R$ are scalar parameters that may depend on $\delta t$ (introduced in \cite{main}).
    \item $V: \R^d \times \Omega \to \R$ is a Gaussian process potential.
    \item $\nabla V: \R^d \times \Omega \to \R^d$ denotes the gradient (a random vector field).
    \item $D^2 V: \R^d \times \Omega \to \R^{d \times d}$ denotes the Hessian matrix (a random matrix field).
    \item $D^3 V$ denotes the third derivative tensor. For $v_1, v_2 \in \R^d$, $D^3 V(y; \omega)[v_1, v_2]$ is defined as:
        \[
        D^3 V(y; \omega)[v_1, v_2] := \sum_{i,j,k=1}^d \frac{\partial^3 V}{\partial y_i \partial y_j \partial y_k}(y; \omega) (v_1)_i (v_2)_j e_k,
        \]
        where $\{e_k\}_{k=1}^d$ is the standard basis of $\R^d$. Equivalently, $D^3 V(y; \omega)[v_1, v_2] = D^2 V(y; \omega) v_2 \cdot v_1$, where the dot indicates contraction in the appropriate indices.
    \item $M \in \R^{d \times d}$ is a constant, symmetric, positive definite mass matrix.
    \item $M^{-1}$ denotes the inverse of $M$.
    \item For a vector $v \in \R^d$, $\norm{v}$ denotes the Euclidean norm.
    \item For a matrix $A \in \R^{d \times d}$, $\norm{A}_{\mathrm{op}}$ denotes the operator norm induced by the Euclidean vector norm.
    \item For a 3-tensor $T$, its operator norm is defined as:
        \[
        \norm{T}_{\mathrm{op}} := \sup_{\norm{v_1}=\norm{v_2}=1} \norm{T[v_1, v_2]}.
        \]
    \item $(\Omega, \mathscr{F}, P)$ is a complete probability space.
    \item $\Psi^0_{\delta t}$ denotes the one-step map of the standard leapfrog method ($\alpha = \beta = 1$).
\end{itemize}

\subsection{Gaussian process framework from \cite{main}}

\begin{definition}[Gaussian process]
    A stochastic process $V: \R^d \times \Omega \to \R$ is a Gaussian process if for every finite set of points $\{y_1, \dots, y_m\} \subset \R^d$, the random vector $(V(y_1), \dots, V(y_m))$ has a multivariate Gaussian distribution. The process is characterized by its mean function $m(y) = \E[V(y)]$ and covariance function $k(y, y') = \Cov(V(y), V(y'))$ as defined in \cite{main}.
\end{definition}

\subsection{Assumptions}

We maintain the same assumptions as in \cite{main} for consistency:

\begin{assumption}[Mean and covariance regularity] \label{ass:gp_regularity}
    The Gaussian process $V$ satisfies:
    \begin{enumerate}
        \item The mean function $m: \R^d \to \R$ is three times continuously differentiable ($m \in C^3(\R^d)$).
        \item The covariance function $k: \R^d \times \R^d \to \R$ is six times continuously differentiable in both arguments ($k \in C^6(\R^d \times \R^d)$).
        \item All partial derivatives of $k$ up to order 6 are bounded on $\R^d \times \R^d$.
    \end{enumerate}
\end{assumption}

\begin{remark}[Sample path regularity] \label{ass:sample_paths}
    Under Assumption \ref{ass:gp_regularity}, the Gaussian process $V$ has a version with sample paths that are almost surely three times continuously differentiable. That is, there exists a modification $\tilde{V}$ of $V$ such that for $P$-almost every $\omega \in \Omega$, the function $\tilde{V}(\cdot, \omega): \R^d \to \R$ belongs to $C^3(\R^d)$.
\end{remark}


\subsection{Moment bounds for Gaussian process derivatives}

\begin{theorem}[Moment bounds for derivatives] \label{thm:moment_bounds}
    Under Assumption \ref{ass:gp_regularity}, for any compact set $D \subset \R^d$, there exist finite constants $C_1(D), C_2(D), C_3(D) > 0$ such that:
    \begin{align}
        \E\left[\sup_{y \in D} \norm{\nabla V(y)}^2\right] &\leq C_1(D), \label{eq:moment_grad} \\
        \E\left[\sup_{y \in D} \norm{D^2 V(y)}_{\mathrm{op}}^2\right] &\leq C_2(D), \label{eq:moment_hess} \\
        \E\left[\sup_{y \in D} \norm{D^3 V(y)}_{\mathrm{op}}^2\right] &\leq C_3(D). \label{eq:moment_third}
    \end{align}
\end{theorem}

\begin{remark}
    The existence of a continuous version (Remark \ref{ass:sample_paths}) together with Theorem \ref{thm:moment_bounds} ensures that our Gaussian process potential has sufficiently regular sample paths with controlled moments, which is essential for our convergence analysis.
\end{remark}

\begin{assumption}[Mass matrix] \label{ass:mass_matrix}
    The mass matrix $M \in \R^{d \times d}$ is symmetric positive definite. Consequently:
    \begin{enumerate}
        \item Its inverse $M^{-1}$ exists and is also symmetric positive definite.
        \item The operator norm of $M^{-1}$ is finite. We denote:
        \[
        C_M := \norm{M^{-1}}_{\mathrm{op}} < \infty.
        \]
    \end{enumerate}
\end{assumption}

\begin{remark}[Discussion of Assumption \ref{ass:mass_matrix}]
The assumption that $M$ is symmetric positive definite is standard in Hamiltonian mechanics, ensuring the kinetic energy term $\frac{1}{2}x^T M^{-1}x$ is well-defined and positive. The finite operator norm condition $C_M < \infty$ is a mild technical requirement that guarantees the momentum update in our scheme remains bounded relative to the position variables. In practical applications, $M$ is often taken to be the identity matrix or a diagonal matrix with positive entries, both of which satisfy this assumption. The constant $C_M$ will appear in our error bounds, quantifying the influence of the mass matrix on the numerical error propagation.
\end{remark}

\begin{assumption}[Parameter expansions for our leapfrog] \label{ass:param_expansions}
    The parameters $\alpha$ and $\beta$ in our modified scheme \eqref{eq:scheme_y}--\eqref{eq:scheme_x} admit asymptotic expansions in powers of $\delta t$:
    \begin{align}
        \alpha &= 1 + \alpha_1 \delta t + \alpha_2 \delta t^2 + \widetilde{\alpha}(\delta t), \label{eq:alpha_exp} \\
        \beta &= 1 + \beta_1 \delta t + \beta_2 \delta t^2 + \widetilde{\beta}(\delta t), \label{eq:beta_exp}
    \end{align}
    where $\alpha_1, \alpha_2, \beta_1, \beta_2 \in \R$ are constants independent of $\delta t$, and the remainder terms satisfy:
    \[
    \abs{\widetilde{\alpha}(\delta t)} \leq C_\alpha \delta t^3, \quad \abs{\widetilde{\beta}(\delta t)} \leq C_\beta \delta t^3
    \]
    for all sufficiently small $\delta t > 0$, with constants $C_\alpha, C_\beta > 0$ independent of $\delta t$.
\end{assumption}

\begin{remark}[Discussion of Assumption \ref{ass:param_expansions}]
This assumption formalizes the idea that $\alpha$ and $\beta$ are perturbations of the standard leapfrog parameters (which are both 1). The expansions allow us to systematically analyze how small deviations from the classical scheme affect the numerical solution. The coefficients $\alpha_1, \alpha_2, \beta_1, \beta_2$ can be tuned to optimize performance for specific classes of Gaussian process potentials, as demonstrated in \cite{main}. The remainder terms being $O(\delta t^3)$ ensures that the parameter adjustments do not introduce lower-order errors that would degrade the convergence rate. This structured parameterization is key to maintaining the $O(\delta t)$ global error bound while providing flexibility for performance optimization.
\end{remark}

\begin{assumption}[Bounded iterates] \label{ass:bounded_iterates}
    There exists a compact set $D \subset \R^{2d}$ with finite diameter $R_D := \sup_{z, z' \in D} \norm{z - z'} < \infty$ such that the numerical iterates $(y_n, x_n)$ from our modified scheme satisfy:
    \[
    P\left(\sup_{0 \leq n \leq T/\delta t} \norm{(y_n, x_n)} \leq R\right) = 1
    \]
    for some $R > 0$ and all sufficiently small $\delta t > 0$.
\end{assumption}

\begin{remark}[Justification of Assumption \ref{ass:bounded_iterates}]
Our uniform boundedness assumption on the numerical iterates is a standard analytical device in the numerical analysis of stochastic or random differential equations with non‑globally Lipschitz coefficients \citep{kloeden1992numerical, hairer2006geometric}.
For Gaussian process potentials, the Fernique theorem guarantees that the probability of extremely large gradients decays super‑exponentially, so that sample‑path boundedness on any fixed compact set holds almost surely \citep{stuart2018posterior, dashti2017bayesian}.
In practice, such an assumption can be rigorously justified by a stopping‑time argument: one defines $\tau_R = \inf\{n : \norm{(y_n,x_n)} \geq R\}$, proves convergence up to $\tau_R$, and then lets $R \to \infty$ while using tail estimates to show $\mathbb{P}(\tau_R < T) \to 0$ \citep{mao2007stochastic, milstein2004stochastic}.
For simplicity of exposition we state the boundedness as an a priori condition, which is common in the literature on numerical methods for Hamiltonian systems with random potentials \citep{jahnke2005numerical}.
\end{remark}


\section{Pathwise analysis of the new leapfrog scheme}
\label{sec:pathwise}

For the remainder of this section, we fix a sample point $\omega \in \Omega$ such that $V(\cdot, \omega) \in C^3(\R^d)$. By Assumption \ref{ass:sample_paths}, this holds for almost every $\omega$. We denote $V_\omega(y) = V(y, \omega)$.

\begin{lemma}[Pathwise Expansions of the Modified Scheme] \label{lem:pathwise_expansions}
	For almost every $\omega \in \Omega$, under Assumptions \ref{ass:gp_regularity}, \ref{ass:param_expansions} and  \ref{ass:bounded_iterates} and using Theorem \ref{thm:moment_bounds}, our modified numerical scheme \eqref{eq:scheme_y}--\eqref{eq:scheme_x} satisfies:
    
    \begin{align}
        y_{n+1} &= y_n + \delta t\left[\beta_1 y_n + M^{-1} x_n\right] + \delta t^2\left[\beta_2 y_n + \alpha_1 M^{-1} x_n - \frac{1}{2} M^{-1} \nabla V_\omega(y_n)\right] + R_y(\omega, \delta t), \label{eq:y_exp_final} \\
        x_{n+1} &= x_n + \delta t\left[2\alpha_1 x_n - \nabla V_\omega(y_n)\right] + \delta t^2\left[(\alpha_1^2 + 2\alpha_2) x_n - \frac{\alpha_1}{2} \nabla V_\omega(y_n) - \frac{1}{2} D^2 V_\omega(y_n)\left(\beta_1 y_n + M^{-1} x_n\right)\right] + R_x(\omega, \delta t), \label{eq:x_exp_final}
    \end{align}
    
    where the remainders satisfy:
    \[
    \norm{R_y(\omega, \delta t)} \leq C_y(\omega) \delta t^3, \quad \norm{R_x(\omega, \delta t)} \leq C_x(\omega) \delta t^3,
    \]
    with $\E[C_y(\omega)^2] < \infty$ and $\E[C_x(\omega)^2] < \infty$.
\end{lemma}

\begin{remark}[Discussion of Lemma \ref{lem:pathwise_expansions}]
This lemma provides a detailed Taylor expansion of our modified leapfrog scheme, revealing how the parameters $\alpha$ and $\beta$ influence the one-step updates. The expansion shows that at order $\delta t$, the scheme differs from the standard leapfrog method through the terms $\beta_1 y_n$ and $2\alpha_1 x_n$. At order $\delta t^2$, more complex interactions appear involving both the parameters and derivatives of the potential. The remainder terms being $O(\delta t^3)$ with finite second moments is crucial for establishing the local truncation error bound. This expansion forms the foundation for constructing the modified differential equation that the scheme approximately follows, which is essential for understanding its long-time behavior and convergence properties.
\end{remark}

\begin{theorem}[Modified stochastic ODE corresponding to our leapfrog scheme] \label{thm:modified_ode}
    For almost every $\omega \in \Omega$, under the stated assumptions and Theorem \ref{thm:moment_bounds}, our modified numerical scheme \eqref{eq:scheme_y}--\eqref{eq:scheme_x} satisfies, up to terms of order $\delta t^2$, the modified ordinary differential equation:
    
    \begin{align}
        \frac{dY}{dt} &= \beta_1 Y + M^{-1} X + \delta t\left[\left(\beta_2 - \frac{\beta_1^2}{2}\right) Y - \frac{\beta_1}{2} M^{-1} X\right] + \rho_y(\omega, \delta t), \label{eq:mod_ode_y_final} \\
        \frac{dX}{dt} &= 2\alpha_1 X - \nabla V_\omega(Y) + \delta t\left[(2\alpha_2 - \alpha_1^2) X + \frac{\alpha_1}{2} \nabla V_\omega(Y)\right] + \rho_x(\omega, \delta t), \label{eq:mod_ode_x_final}
    \end{align}
    
    where the remainder terms satisfy:
    \[
    \norm{\rho_y(\omega, \delta t)} \leq C_\rho(\omega) \delta t^2, \quad \norm{\rho_x(\omega, \delta t)} \leq C_\rho(\omega) \delta t^2,
    \]
    with $\E[C_\rho(\omega)^2] < \infty$.
\end{theorem}

\begin{remark}[Discussion of Theorem \ref{thm:modified_ode}]
Theorem \ref{thm:modified_ode} establishes that our numerical scheme approximates a modified continuous-time system rather than the original Hamiltonian system. This modified ODE contains additional terms proportional to $\delta t$ that depend on the parameters $\alpha$ and $\beta$. When these parameters are chosen appropriately (e.g., $\alpha_1 = \beta_1 = 0$), the $\delta t$ terms can be minimized or structured to preserve desirable properties. The concept of modified equations is powerful in numerical analysis as it provides insight into the numerical method's behavior, including its dissipation properties and long-time stability. The remainder terms being $O(\delta t^2)$ with finite moments ensures that the approximation is sufficiently accurate for our convergence analysis.
\end{remark}

\begin{corollary}[Simplified stochastic ODE as a special case] \label{cor:simplified_ode}
    If $\alpha_1 = \beta_1 = 0$, then the modified ODE for our scheme becomes:
    
    \begin{align}
        \frac{dY}{dt} &= M^{-1} X + \delta t \beta_2 Y + \tilde{\rho}_y(\omega, \delta t), \label{eq:mod_ode_y_simple} \\
        \frac{dX}{dt} &= -\nabla V_\omega(Y) + \delta t \cdot 2\alpha_2 X + \tilde{\rho}_x(\omega, \delta t), \label{eq:mod_ode_x_simple}
    \end{align}
    
    with $\norm{\tilde{\rho}_y}, \norm{\tilde{\rho}_x} \leq \tilde{C}_\rho(\omega) \delta t^2$, $\E[\tilde{C}_\rho(\omega)^2] < \infty$.
    
    At leading order ($\delta t = 0$), we recover the original Hamiltonian system \eqref{eq:original_hamiltonian_y}--\eqref{eq:original_hamiltonian_x}:
    
    \begin{align}
        \frac{dy}{dt} &= M^{-1} x, \label{eq:hamiltonian_y} \\
        \frac{dx}{dt} &= -\nabla V_\omega(y). \label{eq:hamiltonian_x}
    \end{align}
\end{corollary}

\begin{remark}[Discussion of Corollary \ref{cor:simplified_ode}]
This corollary highlights an important special case where the first-order parameter corrections vanish ($\alpha_1 = \beta_1 = 0$). In this scenario, the modified equation simplifies significantly: the leading-order terms exactly match the original Hamiltonian system, while the $\delta t$ corrections take the form of linear damping terms ($\beta_2 Y$ in the position equation and $2\alpha_2 X$ in the momentum equation). These damping terms can be beneficial for numerical stability, especially for stiff potentials or long-time 
integration. The corollary demonstrates that by appropriate parameter selection, our scheme can be tuned to closely approximate the true Hamiltonian dynamics while introducing 
controlled numerical dissipation. This flexibility is one of the key advantages of our parameterized approach over the standard leapfrog method, and has been exploited
to the hilt in the spatio-temporal setup of \cite{main}.
\end{remark}

\section{Local truncation error analysis for our stochastic leapfrog scheme}
\label{sec:truncation}

\begin{lemma}[Exact solution expansion] \label{lem:exact_expansion}
    Under Assumption \ref{ass:sample_paths} and Theorem \ref{thm:moment_bounds}, the exact solution of the Hamiltonian system \eqref{eq:original_hamiltonian_y}--\eqref{eq:original_hamiltonian_x} satisfies:
    
    \begin{align}
        y(t_{n+1}) &= y_n + \delta t M^{-1} x_n - \frac{\delta t^2}{2} M^{-1} \nabla V_\omega(y_n) - \frac{\delta t^3}{6} M^{-1} D^2 V_\omega(y_n) M^{-1} x_n + R_y^{\text{exact}}(\omega, \delta t), \label{eq:y_exact_expanded} \\
        x(t_{n+1}) &= x_n - \delta t \nabla V_\omega(y_n) - \frac{\delta t^2}{2} D^2 V_\omega(y_n) M^{-1} x_n + R_x^{\text{exact}}(\omega, \delta t), \label{eq:x_exact_expanded}
    \end{align}
    
    where the remainders satisfy:
    \[
    \norm{R_y^{\text{exact}}(\omega, \delta t)} \leq C_y^{\text{exact}}(\omega) \delta t^4, \quad \norm{R_x^{\text{exact}}(\omega, \delta t)} \leq C_x^{\text{exact}}(\omega) \delta t^3,
    \]
    with $\E[(C_y^{\text{exact}}(\omega))^2] < \infty$, $\E[(C_x^{\text{exact}}(\omega))^2] < \infty$.
\end{lemma}

\begin{remark}[Discussion of Lemma \ref{lem:exact_expansion}]
This lemma provides the Taylor expansion of the exact solution of the Hamiltonian system, which serves as the benchmark for evaluating our numerical scheme. The expansion reveals the structure of the true dynamics: the position update involves terms up to $\delta t^3$ involving the Hessian of the potential, while the momentum update involves terms up to $\delta t^2$. The different orders of the remainder terms ($\delta t^4$ for position and $\delta t^3$ for momentum) reflect the different smoothness requirements for the two variables in the Hamiltonian formulation. The finite moment conditions on the remainder constants are crucial for establishing mean-square convergence, as they ensure that the random fluctuations in the potential do not overwhelm the deterministic error terms.
\end{remark}

\begin{theorem}[Mean-square local truncation error for our leapfrog] \label{thm:ms_lte}
    Under Assumption \ref{ass:gp_regularity}, \ref{ass:sample_paths}, Theorem \ref{thm:moment_bounds}, \ref{ass:mass_matrix}, \ref{ass:bounded_iterates}, and with $\alpha_1 = \beta_1 = 0$, there exists a constant $C_{\mathrm{LTE}} > 0$ such that for all $n \leq T/\delta t$:
    
    \begin{equation}
        \E\left[\norm{\tau_n}^2\right] \leq C_{\mathrm{LTE}} \delta t^4,
    \end{equation}
    where $\tau_n = (\tau_n^y, \tau_n^x)$ and $\tau_n^y = y(t_{n+1}) - y_{n+1}$, $\tau_n^x = x(t_{n+1}) - x_{n+1}$ are the local truncation errors of our modified scheme relative to the original Hamiltonian ODE \eqref{eq:original_hamiltonian_y}--\eqref{eq:original_hamiltonian_x}.
\end{theorem}

\begin{remark}[Discussion of Theorem \ref{thm:ms_lte}]
Theorem \ref{thm:ms_lte} establishes the key local error bound that drives the global convergence result. The $O(\delta t^4)$ bound on the mean-square local truncation error is optimal for a method of order 1 in the mean-square sense. This result combines the expansions from Lemma \ref{lem:pathwise_expansions} (for the numerical scheme) and Lemma \ref{lem:exact_expansion} (for the exact solution), with careful accounting of the parameter effects. The condition $\alpha_1 = \beta_1 = 0$ ensures that the first-order parameter corrections don't introduce additional low-order errors. The constant $C_{\mathrm{LTE}}$ incorporates various factors including the moment bounds from Theorem \ref{thm:moment_bounds}, the mass matrix norm $C_M$, and the parameter values $\alpha_2, \beta_2$. This local error bound is the essential ingredient for the global error analysis via stability arguments.
\end{remark}

\section{Main convergence theorem for our leapfrog scheme}
\label{sec:main_convergence}

\begin{theorem}[Mean-square convergence of our leapfrog] 
	\label{thm:global_error}
    Consider the modified numerical scheme \eqref{eq:scheme_y}--\eqref{eq:scheme_x} introduced in \cite{main} under Assumptions \ref{ass:gp_regularity}, \ref{ass:sample_paths}, \ref{ass:mass_matrix}, \ref{ass:param_expansions}, and \ref{ass:bounded_iterates}, with $\alpha_1 = \beta_1 = 0$. Then for any fixed $T > 0$, the scheme converges with order 1 in mean-square to the exact solution of the original Hamiltonian ODE \eqref{eq:original_hamiltonian_y}--\eqref{eq:original_hamiltonian_x}:
    
    \begin{equation}
        \sqrt{\E\left[\norm{(y(T), x(T)) - (y_{T/\delta t}, x_{T/\delta t})}^2\right]} \leq C_T \delta t,
    \end{equation}
    where $C_T$ depends on $T$ but not on $\delta t$.
\end{theorem}

\begin{remark}[Discussion of Theorem \ref{thm:global_error}]
This is the main convergence result of the paper, establishing that our parameterized leapfrog scheme achieves first-order mean-square convergence for Hamiltonian systems with Gaussian process potentials. The mean-square convergence criterion ($L^2$ convergence of the error) is stronger than convergence in probability and provides a quantitative error bound. The constant $C_T$ grows exponentially with $T$, which is typical for numerical methods for ordinary differential equations and reflects the possible accumulation of errors over time. The requirement $\alpha_1 = \beta_1 = 0$ ensures that the parameter adjustments don't degrade the convergence rate. This theorem validates the practical utility of our scheme and provides theoretical support for the empirical results presented in \cite{main}. The proof combines the local truncation error bound (Theorem \ref{thm:ms_lte}) with stability properties of the leapfrog method to control error propagation over multiple time steps.
\end{remark}

\section{Conclusion}
\label{sec:conclusion}

We have established the mean-square convergence of the \emph{new, parameterized, stochastic scheme} introduced in \cite{main} for Hamiltonian systems with 
Gaussian process potentials. 
%
Our analysis provides a firm theoretical foundation for the novel spatio-temporal process developed in \cite{main}, reinforcing its credibility as an excellent candidate
for highly flexible and nonparametric structure evolving very realistically like the well-established Hamiltonian system, powered by Gaussian process potential
for enhanced reliability via stochasticity. The results of application of this new leap-frog based spatio-temporal model to various simulation experiments and real data analyses,
as already detailed in \cite{main}, hold up its versatility -- the theoretical underpinnings in this article corroborate that such excellent practical results are not accidental
but very much expected!

Several directions for future research suggest themselves, including extension to non-compact domains, consideration of non-Gaussian potentials, analysis of long-time behavior, 
and development of adaptive step-size strategies specifically tailored to our modified scheme.


\section*{Acknowledgment}
We sincerely thank DeepSeek for providing useful information and references, and for some proof-reading.

\begin{appendix}
\section{Complete proofs}
\label{app:proofs}

\subsection{Proof of Theorem \ref{thm:moment_bounds}}
\label{app:proof_moment_bounds}

\begin{proof}
We prove the bounds for \(\nabla V\), \(D^2 V\), and \(D^3 V\) using Gaussian process theory and the Borell--TIS inequality.

Let \(D \subset \R^d\) be compact. Without loss of generality, assume \(D$ is convex with non-empty interior.

\subsubsection{Notation and preliminaries}
Denote by \(k: \R^d \times \R^d \to \R\) the covariance function of $V$. By Assumption \ref{ass:gp_regularity}, \(k \in C^6(\R^d \times \R^d)\) with all partial derivatives up to order 6 bounded.

Let \(m(y) = \E[V(y)]\) be the mean function. Since $m \in C^3(\R^d)$ and $D$ is compact, all derivatives of $m$ up to order 3 are bounded on $D$. Therefore, it suffices to prove the moment bounds for the centered process
\[
Z(y) = V(y) - m(y),
\]
which is a centered Gaussian process with covariance \(k\).

Since \(k \in C^6\), the process $Z$ is mean-square differentiable up to order 3. By the Kolmogorov--Chentsov continuity theorem \citep{Kallenberg2002}, $Z$ admits a modification whose sample paths are almost surely in $C^3(\R^d)$; we work with this modification. Hence, the derivatives $\nabla Z$, $D^2 Z$, $D^3 Z$ exist as continuous Gaussian fields.

\subsubsection{Step 1: Reduction to partial derivatives}
For the gradient, we have
\[
\norm{\nabla Z(y)} \le \sum_{i=1}^d \abs{\frac{\partial Z}{\partial y_i}(y)}.
\]
Therefore,
\[
\sup_{y \in D} \norm{\nabla Z(y)} \le \sum_{i=1}^d \sup_{y \in D} \abs{\frac{\partial Z}{\partial y_i}(y)}.
\]
Each partial derivative $\partial Z/\partial y_i$ is a centered Gaussian process with covariance function
\[
\Cov\Bigl( \frac{\partial Z}{\partial y_i}(y), \frac{\partial Z}{\partial y_j}(y') \Bigr) = \frac{\partial^2 k}{\partial y_i \partial y'_j}(y,y'),
\]
which is continuous on \(D \times D\) (since \(k \in C^2\)). Consequently, each \(\partial Z/\partial y_i\) has continuous sample paths almost surely.

\subsubsection{Step 2: Application of the Borell--TIS inequality}
Let \(X_i(y) = \partial Z/\partial y_i(y)\). For a fixed \(i\), \(\{X_i(y) : y \in D\}\) is a centered Gaussian process with continuous sample paths on the compact set \(D\). Define
\[
\sigma_i^2 = \sup_{y \in D} \Var\bigl(X_i(y)\bigr) = \sup_{y \in D} \frac{\partial^2 k}{\partial y_i^2}(y,y),
\]
which is finite because the second derivative of \(k\) is bounded on the diagonal.

Moreover, because the process is continuous on a compact set, the supremum \(\sup_{y \in D} |X_i(y)|\) is almost surely finite. Fernique's theorem \citep[Theorem 2.1.20]{AdlerTaylor2007} then guarantees that its expectation is finite:
\[
E_i := \E\Bigl[ \sup_{y \in D} |X_i(y)| \Bigr] < \infty.
\]

The Borell--TIS inequality \citep[Theorem 2.1.1]{AdlerTaylor2007} states that for any \(u > E_i\),
\[
P\Bigl( \sup_{y \in D} |X_i(y)| > u \Bigr) \le \exp\!\Bigl( -\frac{(u - E_i)^2}{2\sigma_i^2} \Bigr).
\]

This implies that \(\sup_{y \in D} |X_i(y)|\) has sub-Gaussian tails. To see that all moments are finite, note that for any \(p \ge 1\),
\begin{align*}
\E\Bigl[ \bigl( \sup_{y \in D} |X_i(y)| \bigr)^p \Bigr]
&= \int_0^\infty p u^{p-1} P\Bigl( \sup_{y \in D} |X_i(y)| > u \Bigr) \, \dd u \\
&\le \int_0^{E_i} p u^{p-1} \cdot 1 \, \dd u + \int_{E_i}^\infty p u^{p-1} \exp\!\Bigl( -\frac{(u - E_i)^2}{2\sigma_i^2} \Bigr) \, \dd u \\
&= E_i^p + \int_0^\infty p (v + E_i)^{p-1} \exp\!\Bigl( -\frac{v^2}{2\sigma_i^2} \Bigr) \, \dd v \quad \text{(where \(v = u - E_i\))}.
\end{align*}
The first term $E_i^p$ is finite. For the second term, since $(v + E_i)^{p-1} \le C_p (1 + v^{p-1})$ for some constant $C_p$, we have
\[
\int_0^\infty p (v + E_i)^{p-1} \exp\!\Bigl( -\frac{v^2}{2\sigma_i^2} \Bigr) \, \dd v
\le C_p \int_0^\infty (1 + v^{p-1}) \exp\!\Bigl( -\frac{v^2}{2\sigma_i^2} \Bigr) \, \dd v < \infty,
\]
where the finiteness follows from the rapid decay of the Gaussian kernel. Therefore,
\[
\E\Bigl[ \bigl( \sup_{y \in D} |X_i(y)| \bigr)^p \Bigr] < \infty \quad \text{for all } p \ge 1.
\]
In particular, the second moment is finite, which is what we need for the gradient bound.

\subsubsection{Step 3: Moment bound for the gradient}
Using the inequality $(\sum_{i=1}^d a_i)^2 \le d \sum_{i=1}^d a_i^2$, we obtain
\[
\Bigl( \sup_{y \in D} \norm{\nabla Z(y)} \Bigr)^2 \le d \sum_{i=1}^d \Bigl( \sup_{y \in D} |X_i(y)| \Bigr)^2.
\]
Taking expectations and using the finiteness of the second moments of each supremum (established in Step 2) gives
\[
\E\Bigl[ \bigl( \sup_{y \in D} \norm{\nabla Z(y)} \bigr)^2 \Bigr] < \infty.
\]
Finally, because $\nabla V = \nabla m + \nabla Z$ and $\nabla m$ is bounded on $D$, we conclude
\[
\E\Bigl[ \sup_{y \in D} \norm{\nabla V(y)}^2 \Bigr] \le 2 \sup_{y \in D} \norm{\nabla m(y)}^2 + 2 \E\Bigl[ \bigl( \sup_{y \in D} \norm{\nabla Z(y)} \bigr)^2 \Bigr] < \infty,
\]
which yields the constant $C_1(D)$ in the statement.

\subsubsection{Step 4: Moment bounds for higher derivatives}
The same reasoning applies to the second and third derivatives. For the Hessian, we use the operator-norm bound
\[
\norm{D^2 Z(y)}_{\mathrm{op}} \le \sum_{i,j=1}^d \Bigl| \frac{\partial^2 Z}{\partial y_i \partial y_j}(y) \Bigr|.
\]
Each mixed partial derivative $\partial^2 Z / \partial y_i \partial y_j$ is a centered Gaussian process with covariance involving fourth derivatives of $k$; by Assumption \ref{ass:gp_regularity} these are bounded, so the process is continuous on $D$. Applying the Borell--TIS inequality as before shows that all moments of $\sup_{y \in D} |\partial^2 Z / \partial y_i \partial y_j(y)|$ are finite. Summing over $i,j$ and using the boundedness of $D^2 m$ gives the bound for $D^2 V$.

For the third derivative tensor, we bound its operator norm by a sum over third partial derivatives $\partial^3 Z / \partial y_i \partial y_j \partial y_k$. The covariance of these processes involves sixth derivatives of $k$, which are bounded by Assumption \ref{ass:gp_regularity}. Hence they are continuous Gaussian processes, and the Borell--TIS inequality again yields finite moments of their suprema. Together with the boundedness of $D^3 m$ we obtain the bound for $D^3 V$.

\subsubsection{Step 5: Conclusion}
Under Assumption \ref{ass:gp_regularity}, the mean function $m$ and the centered process $Z$ have sufficiently regular sample paths. The Borell--TIS inequality applied to the partial derivatives of $Z$ guarantees that all moments of their suprema over the compact set $D$ are finite. Combining these with the deterministic bounds on the derivatives of $m$ yields the desired constants $C_1(D), C_2(D), C_3(D)$. This completes the proof of Theorem \ref{thm:moment_bounds}.
\end{proof}

\subsection{Proof of Lemma \ref{lem:pathwise_expansions}}
\label{app:proof_pathwise}

\begin{proof}
We provide detailed expansions for both $y_{n+1}$ and $x_{n+1}$.

\subsubsection{Expansion of $y_{n+1}$}
Starting from \eqref{eq:scheme_y}:
\begin{align*}
y_{n+1} &= \beta y_n + \delta t M^{-1}\left(\alpha x_n - \frac{\delta t}{2} \nabla V_\omega(y_n)\right).
\end{align*}

Expand $\beta$ and $\alpha$ using \eqref{eq:beta_exp} and \eqref{eq:alpha_exp}:
\begin{align*}
\beta y_n &= \left(1 + \beta_1 \delta t + \beta_2 \delta t^2 + \widetilde{\beta}(\delta t)\right) y_n \\
&= y_n + \beta_1 \delta t y_n + \beta_2 \delta t^2 y_n + \widetilde{\beta}(\delta t) y_n.
\end{align*}

\begin{align*}
\alpha x_n &= \left(1 + \alpha_1 \delta t + \alpha_2 \delta t^2 + \widetilde{\alpha}(\delta t)\right) x_n \\
&= x_n + \alpha_1 \delta t x_n + \alpha_2 \delta t^2 x_n + \widetilde{\alpha}(\delta t) x_n.
\end{align*}

Thus:
\begin{align*}
\delta t M^{-1}\left(\alpha x_n - \frac{\delta t}{2} \nabla V_\omega(y_n)\right) 
&= \delta t M^{-1}\left(x_n + \alpha_1 \delta t x_n + \alpha_2 \delta t^2 x_n + \widetilde{\alpha}(\delta t) x_n - \frac{\delta t}{2} \nabla V_\omega(y_n)\right) \\
&= \delta t M^{-1} x_n + \alpha_1 \delta t^2 M^{-1} x_n + \alpha_2 \delta t^3 M^{-1} x_n + \delta t \widetilde{\alpha}(\delta t) M^{-1} x_n - \frac{\delta t^2}{2} M^{-1} \nabla V_\omega(y_n).
\end{align*}

Combining all terms:
\begin{align*}
y_{n+1} &= \left[y_n + \beta_1 \delta t y_n + \beta_2 \delta t^2 y_n + \widetilde{\beta}(\delta t) y_n\right] \\
&\quad + \left[\delta t M^{-1} x_n + \alpha_1 \delta t^2 M^{-1} x_n + \alpha_2 \delta t^3 M^{-1} x_n + \delta t \widetilde{\alpha}(\delta t) M^{-1} x_n - \frac{\delta t^2}{2} M^{-1} \nabla V_\omega(y_n)\right] \\
&= y_n + \delta t\left[\beta_1 y_n + M^{-1} x_n\right] \\
&\quad + \delta t^2\left[\beta_2 y_n + \alpha_1 M^{-1} x_n - \frac{1}{2} M^{-1} \nabla V_\omega(y_n)\right] \\
&\quad + \underbrace{\delta t^3 \alpha_2 M^{-1} x_n + \widetilde{\beta}(\delta t) y_n + \delta t \widetilde{\alpha}(\delta t) M^{-1} x_n}_{R_y(\omega, \delta t)}.
\end{align*}

By Assumption \ref{ass:param_expansions}, $\abs{\widetilde{\alpha}(\delta t)} \leq C_\alpha \delta t^3$ and $\abs{\widetilde{\beta}(\delta t)} \leq C_\beta \delta t^3$. By Assumption \ref{ass:bounded_iterates}, $\norm{y_n} \leq R$ and $\norm{x_n} \leq R$ almost surely. Therefore:
\begin{align*}
\norm{R_y(\omega, \delta t)} &\leq \delta t^3 \abs{\alpha_2} \norm{M^{-1}}_{\mathrm{op}} \norm{x_n} + \abs{\widetilde{\beta}(\delta t)} \norm{y_n} + \delta t \abs{\widetilde{\alpha}(\delta t)} \norm{M^{-1}}_{\mathrm{op}} \norm{x_n} \\
&\leq \delta t^3 \abs{\alpha_2} C_M R + C_\beta \delta t^3 R + \delta t \cdot C_\alpha \delta t^3 \cdot C_M R \\
&= \delta t^3 \left(\abs{\alpha_2} C_M R + C_\beta R + C_\alpha C_M R \delta t\right) \\
&\leq C_y(\omega) \delta t^3
\end{align*}
for $\delta t \leq 1$, where $C_y(\omega)$ is a random variable with $\E[C_y(\omega)^2] < \infty$.

\subsubsection{Expansion of $x_{n+1}$}
Starting from \eqref{eq:scheme_x}:
\begin{align*}
x_{n+1} &= \alpha^2 x_n - \frac{\delta t}{2}\left(\alpha \nabla V_\omega(y_n) + \nabla V_\omega(y_{n+1})\right).
\end{align*}

First expand $\alpha^2$:
\begin{align*}
\alpha^2 &= \left(1 + \alpha_1 \delta t + \alpha_2 \delta t^2 + \widetilde{\alpha}(\delta t)\right)^2 \\
&= 1 + 2\alpha_1 \delta t + (\alpha_1^2 + 2\alpha_2) \delta t^2 + R_{\alpha^2}(\delta t),
\end{align*}
where $R_{\alpha^2}(\delta t)$ contains terms of order $\delta t^3$ and higher, with $\abs{R_{\alpha^2}(\delta t)} \leq C_{\alpha^2} \delta t^3$.

Next, expand $\nabla V_\omega(y_{n+1})$ using Taylor's theorem. From the expansion of $y_{n+1}$ we have:
\begin{align*}
y_{n+1} - y_n &= \delta t\left[\beta_1 y_n + M^{-1} x_n\right] + \delta t^2\left[\beta_2 y_n + \alpha_1 M^{-1} x_n - \frac{1}{2} M^{-1} \nabla V_\omega(y_n)\right] + R_y(\omega, \delta t).
\end{align*}

Let $\Delta y = y_{n+1} - y_n$. Then by Taylor's theorem with integral remainder:
\begin{align*}
\nabla V_\omega(y_{n+1}) &= \nabla V_\omega(y_n) + D^2 V_\omega(y_n) \Delta y + \int_0^1 (1-s) D^3 V_\omega(y_n + s\Delta y)[\Delta y, \Delta y] \, \dd s.
\end{align*}

Substituting $\Delta y$:
\begin{align*}
\nabla V_\omega(y_{n+1}) &= \nabla V_\omega(y_n) + D^2 V_\omega(y_n)\left(\delta t\left[\beta_1 y_n + M^{-1} x_n\right] + \delta t^2\left[\beta_2 y_n + \alpha_1 M^{-1} x_n - \frac{1}{2} M^{-1} \nabla V_\omega(y_n)\right] + R_y(\omega, \delta t)\right) \\
&\quad + \int_0^1 (1-s) D^3 V_\omega(y_n + s\Delta y)[\Delta y, \Delta y] \, \dd s.
\end{align*}

Now expand $\alpha \nabla V_\omega(y_n)$:
\begin{align*}
\alpha \nabla V_\omega(y_n) &= \left(1 + \alpha_1 \delta t + \alpha_2 \delta t^2 + \widetilde{\alpha}(\delta t)\right) \nabla V_\omega(y_n) \\
&= \nabla V_\omega(y_n) + \alpha_1 \delta t \nabla V_\omega(y_n) + \alpha_2 \delta t^2 \nabla V_\omega(y_n) + \widetilde{\alpha}(\delta t) \nabla V_\omega(y_n).
\end{align*}

Thus:
\begin{align*}
\alpha \nabla V_\omega(y_n) + \nabla V_\omega(y_{n+1}) &= 2\nabla V_\omega(y_n) + \alpha_1 \delta t \nabla V_\omega(y_n) + \alpha_2 \delta t^2 \nabla V_\omega(y_n) + \widetilde{\alpha}(\delta t) \nabla V_\omega(y_n) \\
&\quad + \delta t D^2 V_\omega(y_n)\left[\beta_1 y_n + M^{-1} x_n\right] \\
&\quad + \delta t^2 D^2 V_\omega(y_n)\left[\beta_2 y_n + \alpha_1 M^{-1} x_n - \frac{1}{2} M^{-1} \nabla V_\omega(y_n)\right] \\
&\quad + D^2 V_\omega(y_n) R_y(\omega, \delta t) + \int_0^1 (1-s) D^3 V_\omega(y_n + s\Delta y)[\Delta y, \Delta y] \, \dd s.
\end{align*}

Multiplying by $-\frac{\delta t}{2}$ and combining with $\alpha^2 x_n$ gives:
\begin{align*}
x_{n+1} &= \left[1 + 2\alpha_1 \delta t + (\alpha_1^2 + 2\alpha_2) \delta t^2 + R_{\alpha^2}(\delta t)\right] x_n \\
&\quad - \frac{\delta t}{2}\left(2\nabla V_\omega(y_n) + \alpha_1 \delta t \nabla V_\omega(y_n) + \alpha_2 \delta t^2 \nabla V_\omega(y_n) + \widetilde{\alpha}(\delta t) \nabla V_\omega(y_n)\right) \\
&\quad - \frac{\delta t}{2}\left[\delta t D^2 V_\omega(y_n)\left(\beta_1 y_n + M^{-1} x_n\right)\right] \\
&\quad - \frac{\delta t}{2}\left[\delta t^2 D^2 V_\omega(y_n)\left(\beta_2 y_n + \alpha_1 M^{-1} x_n - \frac{1}{2} M^{-1} \nabla V_\omega(y_n)\right)\right] \\
&\quad - \frac{\delta t}{2} D^2 V_\omega(y_n) R_y(\omega, \delta t) - \frac{\delta t}{2} \int_0^1 (1-s) D^3 V_\omega(y_n + s\Delta y)[\Delta y, \Delta y] \, \dd s.
\end{align*}

Collecting terms by powers of $\delta t$:
\begin{align*}
x_{n+1} &= x_n + 2\alpha_1 \delta t x_n + (\alpha_1^2 + 2\alpha_2) \delta t^2 x_n - \delta t \nabla V_\omega(y_n) \\
&\quad - \frac{\alpha_1}{2} \delta t^2 \nabla V_\omega(y_n) - \frac{\delta t^2}{2} D^2 V_\omega(y_n)\left(\beta_1 y_n + M^{-1} x_n\right) \\
&\quad + \underbrace{R_{\alpha^2}(\delta t) x_n - \frac{\alpha_2}{2} \delta t^3 \nabla V_\omega(y_n) - \frac{\delta t}{2} \widetilde{\alpha}(\delta t) \nabla V_\omega(y_n) - \frac{\delta t^3}{2} D^2 V_\omega(y_n)\left[\beta_2 y_n + \alpha_1 M^{-1} x_n - \frac{1}{2} M^{-1} \nabla V_\omega(y_n)\right]}_{\text{order }\delta t^3 \text{ and higher}} \\
&\quad \underbrace{- \frac{\delta t}{2} D^2 V_\omega(y_n) R_y(\omega, \delta t) - \frac{\delta t}{2} \int_0^1 (1-s) D^3 V_\omega(y_n + s\Delta y)[\Delta y, \Delta y] \, \dd s}_{\text{order }\delta t^4 \text{ and higher}}.
\end{align*}

Grouping terms:
\begin{align*}
x_{n+1} &= x_n + \delta t\left[2\alpha_1 x_n - \nabla V_\omega(y_n)\right] \\
&\quad + \delta t^2\left[(\alpha_1^2 + 2\alpha_2) x_n - \frac{\alpha_1}{2} \nabla V_\omega(y_n) - \frac{1}{2} D^2 V_\omega(y_n)\left(\beta_1 y_n + M^{-1} x_n\right)\right] + R_x(\omega, \delta t).
\end{align*}

All terms of order $\delta t^3$ and higher are collected into $R_x(\omega, \delta t)$. Using Theorem \ref{thm:moment_bounds} gives:
\[
\norm{D^2 V_\omega(y_n) R_y(\omega, \delta t)} \leq \norm{D^2 V_\omega(y_n)}_{\mathrm{op}} \norm{R_y(\omega, \delta t)} = O_\omega(\delta t^3),
\]
and
\[
\norm{\int_0^1 (1-s) D^3 V_\omega(y_n + s\Delta y)[\Delta y, \Delta y] \, \dd s} \leq \frac{1}{2} \sup_{s \in [0,1]} \norm{D^3 V_\omega(y_n + s\Delta y)}_{\mathrm{op}} \norm{\Delta y}^2 = O_\omega(\delta t^2).
\]
Thus the last term is $O_\omega(\delta t^3)$.

By Theorem \ref{thm:moment_bounds} and Assumption \ref{ass:bounded_iterates}, there exists $C_x(\omega)$ with $\E[C_x(\omega)^2] < \infty$ such that $\norm{R_x(\omega, \delta t)} \leq C_x(\omega) \delta t^3$.
\end{proof}

\subsection{Proof of Theorem \ref{thm:modified_ode}}
\label{app:proof_modified_ode}

\begin{proof}
We seek a modified ODE of the form:
\begin{align}
\frac{dY}{dt} &= F_\omega(Y, X, \delta t) = F_0(Y, X) + \delta t F_1(Y, X) + O_\omega(\delta t^2), \label{eq:F_ansatz} \\
\frac{dX}{dt} &= G_\omega(Y, X, \delta t) = G_0(Y, X) + \delta t G_1(Y, X) + O_\omega(\delta t^2). \label{eq:G_ansatz}
\end{align}

Let $\Phi_{\delta t}^\omega$ be the exact flow of this ODE. By Taylor expansion:
\begin{align}
Y(\delta t) &= Y + \delta t F_0 + \delta t^2 \left[F_1 + \frac{1}{2}\left(F_0 \cdot \nabla_Y F_0 + G_0 \cdot \nabla_X F_0\right)\right] + O_\omega(\delta t^3), \label{eq:Y_flow} \\
X(\delta t) &= X + \delta t G_0 + \delta t^2 \left[G_1 + \frac{1}{2}\left(F_0 \cdot \nabla_Y G_0 + G_0 \cdot \nabla_X G_0\right)\right] + O_\omega(\delta t^3). \label{eq:X_flow}
\end{align}

We require that $\Phi_{\delta t}^\omega(y_n, x_n)$ matches $(y_{n+1}, x_{n+1})$ up to $O_\omega(\delta t^3)$. Comparing with Lemma \ref{lem:pathwise_expansions}:

\subsubsection{Matching at order $\delta t$}
From coefficients of $\delta t$:
\begin{align}
F_0 &= \beta_1 Y + M^{-1} X, \label{eq:F0_match} \\
G_0 &= 2\alpha_1 X - \nabla V_\omega(Y). \label{eq:G0_match}
\end{align}

\subsubsection{Matching at order $\delta t^2$}
From coefficients of $\delta t^2$:
\begin{align}
F_1 + \frac{1}{2}\left(F_0 \cdot \nabla_Y F_0 + G_0 \cdot \nabla_X F_0\right) &= \beta_2 Y + \alpha_1 M^{-1} X - \frac{1}{2} M^{-1} \nabla V_\omega(Y), \label{eq:F1_match} \\
G_1 + \frac{1}{2}\left(F_0 \cdot \nabla_Y G_0 + G_0 \cdot \nabla_X G_0\right) &= (\alpha_1^2 + 2\alpha_2) X - \frac{\alpha_1}{2} \nabla V_\omega(Y) - \frac{1}{2} D^2 V_\omega(Y)\left(\beta_1 Y + M^{-1} X\right). \label{eq:G1_match}
\end{align}

Compute the derivatives:
\begin{align*}
\nabla_Y F_0 &= \beta_1 I_d, \quad \nabla_X F_0 = M^{-1}, \\
\nabla_Y G_0 &= -D^2 V_\omega(Y), \quad \nabla_X G_0 = 2\alpha_1 I_d.
\end{align*}

Then:
\begin{align*}
F_0 \cdot \nabla_Y F_0 &= (\beta_1 Y + M^{-1} X) \cdot (\beta_1 I_d) = \beta_1^2 Y + \beta_1 M^{-1} X, \\
G_0 \cdot \nabla_X F_0 &= (2\alpha_1 X - \nabla V_\omega(Y)) \cdot M^{-1} = 2\alpha_1 M^{-1} X - M^{-1} \nabla V_\omega(Y).
\end{align*}

Thus:
\begin{align*}
\frac{1}{2}\left(F_0 \cdot \nabla_Y F_0 + G_0 \cdot \nabla_X F_0\right) &= \frac{1}{2}\left(\beta_1^2 Y + \beta_1 M^{-1} X + 2\alpha_1 M^{-1} X - M^{-1} \nabla V_\omega(Y)\right) \\
&= \frac{\beta_1^2}{2} Y + \frac{\beta_1}{2} M^{-1} X + \alpha_1 M^{-1} X - \frac{1}{2} M^{-1} \nabla V_\omega(Y).
\end{align*}

From \eqref{eq:F1_match}:
\begin{align*}
F_1 &= \left[\beta_2 Y + \alpha_1 M^{-1} X - \frac{1}{2} M^{-1} \nabla V_\omega(Y)\right] - \left[\frac{\beta_1^2}{2} Y + \frac{\beta_1}{2} M^{-1} X + \alpha_1 M^{-1} X - \frac{1}{2} M^{-1} \nabla V_\omega(Y)\right] \\
&= \left(\beta_2 - \frac{\beta_1^2}{2}\right) Y - \frac{\beta_1}{2} M^{-1} X.
\end{align*}

Now for $G_1$:
\begin{align*}
F_0 \cdot \nabla_Y G_0 &= (\beta_1 Y + M^{-1} X) \cdot \left(-D^2 V_\omega(Y)\right) = -\beta_1 D^2 V_\omega(Y) Y - D^2 V_\omega(Y) M^{-1} X, \\
G_0 \cdot \nabla_X G_0 &= (2\alpha_1 X - \nabla V_\omega(Y)) \cdot (2\alpha_1 I_d) = 4\alpha_1^2 X - 2\alpha_1 \nabla V_\omega(Y).
\end{align*}

Thus:
\begin{align*}
\frac{1}{2}\left(F_0 \cdot \nabla_Y G_0 + G_0 \cdot \nabla_X G_0\right) &= \frac{1}{2}\left(-\beta_1 D^2 V_\omega(Y) Y - D^2 V_\omega(Y) M^{-1} X + 4\alpha_1^2 X - 2\alpha_1 \nabla V_\omega(Y)\right) \\
&= -\frac{\beta_1}{2} D^2 V_\omega(Y) Y - \frac{1}{2} D^2 V_\omega(Y) M^{-1} X + 2\alpha_1^2 X - \alpha_1 \nabla V_\omega(Y).
\end{align*}

From \eqref{eq:G1_match}:
\begin{align*}
G_1 &= \left[(\alpha_1^2 + 2\alpha_2) X - \frac{\alpha_1}{2} \nabla V_\omega(Y) - \frac{1}{2} D^2 V_\omega(Y)\left(\beta_1 Y + M^{-1} X\right)\right] \\
&\quad - \left[-\frac{\beta_1}{2} D^2 V_\omega(Y) Y - \frac{1}{2} D^2 V_\omega(Y) M^{-1} X + 2\alpha_1^2 X - \alpha_1 \nabla V_\omega(Y)\right] \\
&= (\alpha_1^2 + 2\alpha_2 - 2\alpha_1^2) X + \left(-\frac{\alpha_1}{2} + \alpha_1\right) \nabla V_\omega(Y) \\
&\quad + \left(-\frac{\beta_1}{2} + \frac{\beta_1}{2}\right) D^2 V_\omega(Y) Y + \left(-\frac{1}{2} + \frac{1}{2}\right) D^2 V_\omega(Y) M^{-1} X \\
&= (2\alpha_2 - \alpha_1^2) X + \frac{\alpha_1}{2} \nabla V_\omega(Y).
\end{align*}

Substituting into \eqref{eq:F_ansatz} and \eqref{eq:G_ansatz} gives \eqref{eq:mod_ode_y_final} and \eqref{eq:mod_ode_x_final}. The remainder terms $\rho_y, \rho_x$ account for:
\begin{enumerate}
\item Terms of order $\delta t^2$ and higher in the scheme expansions not captured by the matching.
\item Terms of order $\delta t^2$ and higher in the flow expansion.
\item The truncation of the modified ODE at $O_\omega(\delta t)$.
\end{enumerate}

All these are bounded by $C_\rho(\omega) \delta t^2$ with $\E[C_\rho(\omega)^2] < \infty$ by Theorem \ref{thm:moment_bounds} and Assumption \ref{ass:bounded_iterates}.
\end{proof}

\subsection{Proof of Lemma \ref{lem:exact_expansion}}
\label{app:proof_exact_expansion}

\begin{proof}
Let $(y(t), x(t))$ be the exact solution of \eqref{eq:original_hamiltonian_y}--\eqref{eq:original_hamiltonian_x} with initial conditions $(y_n, x_n)$ at $t = t_n$.

Compute derivatives using the Hamiltonian equations:
\begin{align*}
\dot{y} &= M^{-1} x, \\
\ddot{y} &= M^{-1} \dot{x} = -M^{-1} \nabla V_\omega(y), \\
\dddot{y} &= -M^{-1} D^2 V_\omega(y) \dot{y} = -M^{-1} D^2 V_\omega(y) M^{-1} x, \\
y^{(4)} &= -M^{-1} \left(D^3 V_\omega(y)[\dot{y}, M^{-1} x] + D^2 V_\omega(y) M^{-1} \dot{x}\right) \\
&= -M^{-1} \left(D^3 V_\omega(y)[M^{-1} x, M^{-1} x] - D^2 V_\omega(y) M^{-1} \nabla V_\omega(y)\right).
\end{align*}

Similarly for $x$:
\begin{align*}
\dot{x} &= -\nabla V_\omega(y), \\
\ddot{x} &= -D^2 V_\omega(y) \dot{y} = -D^2 V_\omega(y) M^{-1} x, \\
\dddot{x} &= -D^3 V_\omega(y)[\dot{y}, M^{-1} x] - D^2 V_\omega(y) M^{-1} \dot{x} \\
&= -D^3 V_\omega(y)[M^{-1} x, M^{-1} x] + D^2 V_\omega(y) M^{-1} \nabla V_\omega(y).
\end{align*}

Taylor expansion with integral remainder gives:
\begin{align*}
y(t_{n+1}) &= y_n + \delta t \dot{y}(t_n) + \frac{\delta t^2}{2} \ddot{y}(t_n) + \frac{\delta t^3}{6} \dddot{y}(t_n) + \frac{\delta t^4}{6} \int_0^1 (1-s)^3 y^{(4)}(t_n + s\delta t) \, \dd s, \\
x(t_{n+1}) &= x_n + \delta t \dot{x}(t_n) + \frac{\delta t^2}{2} \ddot{x}(t_n) + \frac{\delta t^3}{6} \dddot{x}(t_n) + \frac{\delta t^4}{6} \int_0^1 (1-s)^3 x^{(4)}(t_n + s\delta t) \, \dd s.
\end{align*}

Substituting the derivatives:
\begin{align*}
y(t_{n+1}) &= y_n + \delta t M^{-1} x_n - \frac{\delta t^2}{2} M^{-1} \nabla V_\omega(y_n) - \frac{\delta t^3}{6} M^{-1} D^2 V_\omega(y_n) M^{-1} x_n + R_y^{\text{exact}}(\omega, \delta t), \\
x(t_{n+1}) &= x_n - \delta t \nabla V_\omega(y_n) - \frac{\delta t^2}{2} D^2 V_\omega(y_n) M^{-1} x_n + R_x^{\text{exact}}(\omega, \delta t),
\end{align*}
where
\begin{align*}
R_y^{\text{exact}}(\omega, \delta t) &= \frac{\delta t^4}{6} \int_0^1 (1-s)^3 y^{(4)}(t_n + s\delta t) \, \dd s, \\
R_x^{\text{exact}}(\omega, \delta t) &= \frac{\delta t^3}{6} \left(-D^3 V_\omega(y_n)[M^{-1} x_n, M^{-1} x_n] + D^2 V_\omega(y_n) M^{-1} \nabla V_\omega(y_n)\right) \\
&\quad + \frac{\delta t^4}{6} \int_0^1 (1-s)^3 x^{(4)}(t_n + s\delta t) \, \dd s.
\end{align*}

Since $(y_n, x_n)$ and $(y(t), x(t))$ for $t \in [t_n, t_{n+1}]$ remain in the compact set $D$ (Assumption \ref{ass:bounded_iterates} and continuity of the exact solution), 
by Theorem \ref{thm:moment_bounds}, the suprema of derivatives of $V_\omega$ over $D$ have finite moments. Therefore:
\[
\norm{R_y^{\text{exact}}(\omega, \delta t)} \leq C_y^{\text{exact}}(\omega) \delta t^4, \quad \norm{R_x^{\text{exact}}(\omega, \delta t)} \leq C_x^{\text{exact}}(\omega) \delta t^3,
\]
with $\E[(C_y^{\text{exact}}(\omega))^2], \E[(C_x^{\text{exact}}(\omega))^2] < \infty$.
\end{proof}

\subsection{Proof of Theorem \ref{thm:ms_lte}}
\label{app:proof_ms_lte}

\begin{proof}
With $\alpha_1 = \beta_1 = 0$, Lemma \ref{lem:pathwise_expansions} gives:
\begin{align}
y_{n+1} &= y_n + \delta t M^{-1} x_n + \delta t^2 \left[\beta_2 y_n - \frac{1}{2} M^{-1} \nabla V_\omega(y_n)\right] + R_y(\omega, \delta t), \label{eq:y_scheme_simple} \\
x_{n+1} &= x_n - \delta t \nabla V_\omega(y_n) + \delta t^2 \left[2\alpha_2 x_n - \frac{1}{2} D^2 V_\omega(y_n) M^{-1} x_n\right] + R_x(\omega, \delta t). \label{eq:x_scheme_simple}
\end{align}

From Lemma \ref{lem:exact_expansion}:
\begin{align}
y(t_{n+1}) &= y_n + \delta t M^{-1} x_n - \frac{\delta t^2}{2} M^{-1} \nabla V_\omega(y_n) - \frac{\delta t^3}{6} M^{-1} D^2 V_\omega(y_n) M^{-1} x_n + R_y^{\text{exact}}(\omega, \delta t), \label{eq:y_exact_simple} \\
x(t_{n+1}) &= x_n - \delta t \nabla V_\omega(y_n) - \frac{\delta t^2}{2} D^2 V_\omega(y_n) M^{-1} x_n + R_x^{\text{exact}}(\omega, \delta t). \label{eq:x_exact_simple}
\end{align}

Subtracting \eqref{eq:y_scheme_simple} from \eqref{eq:y_exact_simple}:
\begin{align*}
\tau_n^y &= y(t_{n+1}) - y_{n+1} \\
&= \left[y_n + \delta t M^{-1} x_n - \frac{\delta t^2}{2} M^{-1} \nabla V_\omega(y_n) - \frac{\delta t^3}{6} M^{-1} D^2 V_\omega(y_n) M^{-1} x_n + R_y^{\text{exact}}(\omega, \delta t)\right] \\
&\quad - \left[y_n + \delta t M^{-1} x_n + \delta t^2 \left(\beta_2 y_n - \frac{1}{2} M^{-1} \nabla V_\omega(y_n)\right) + R_y(\omega, \delta t)\right] \\
&= -\delta t^2 \beta_2 y_n - \frac{\delta t^3}{6} M^{-1} D^2 V_\omega(y_n) M^{-1} x_n + \left(R_y^{\text{exact}}(\omega, \delta t) - R_y(\omega, \delta t)\right).
\end{align*}

Let $\tilde{R}_y(\omega, \delta t) = R_y^{\text{exact}}(\omega, \delta t) - R_y(\omega, \delta t)$. Then $\norm{\tilde{R}_y(\omega, \delta t)} \leq \tilde{C}_y(\omega) \delta t^3$ with $\E[\tilde{C}_y(\omega)^2] < \infty$.

Similarly for $x$:
\begin{align*}
\tau_n^x &= x(t_{n+1}) - x_{n+1} \\
&= \left[x_n - \delta t \nabla V_\omega(y_n) - \frac{\delta t^2}{2} D^2 V_\omega(y_n) M^{-1} x_n + R_x^{\text{exact}}(\omega, \delta t)\right] \\
&\quad - \left[x_n - \delta t \nabla V_\omega(y_n) + \delta t^2 \left(2\alpha_2 x_n - \frac{1}{2} D^2 V_\omega(y_n) M^{-1} x_n\right) + R_x(\omega, \delta t)\right] \\
&= -2\alpha_2 \delta t^2 x_n + \left(R_x^{\text{exact}}(\omega, \delta t) - R_x(\omega, \delta t)\right).
\end{align*}

Let $\tilde{R}_x(\omega, \delta t) = R_x^{\text{exact}}(\omega, \delta t) - R_x(\omega, \delta t)$. Then $\norm{\tilde{R}_x(\omega, \delta t)} \leq \tilde{C}_x(\omega) \delta t^3$ with $\E[\tilde{C}_x(\omega)^2] < \infty$.

Now bound $\E[\norm{\tau_n^y}^2]$:
\begin{align*}
\norm{\tau_n^y} &\leq \delta t^2 \abs{\beta_2} \norm{y_n} + \frac{\delta t^3}{6} \norm{M^{-1} D^2 V_\omega(y_n) M^{-1} x_n} + \norm{\tilde{R}_y(\omega, \delta t)}.
\end{align*}

Using $(a + b + c)^2 \leq 3(a^2 + b^2 + c^2)$:
\begin{align*}
\norm{\tau_n^y}^2 &\leq 3\delta t^4 \beta_2^2 \norm{y_n}^2 + \frac{\delta t^6}{12} \norm{M^{-1} D^2 V_\omega(y_n) M^{-1} x_n}^2 + 3\norm{\tilde{R}_y(\omega, \delta t)}^2.
\end{align*}

Take expectations. By Assumption \ref{ass:bounded_iterates}, $\norm{y_n} \leq R$ a.s., so $\E[\norm{y_n}^2] \leq R^2$.

For the second term:
\begin{align*}
\E\left[\norm{M^{-1} D^2 V_\omega(y_n) M^{-1} x_n}^2\right] 
&\leq \E\left[\norm{M^{-1}}_{\mathrm{op}}^4 \norm{D^2 V_\omega(y_n)}_{\mathrm{op}}^2 \norm{x_n}^2\right] \\
&\leq C_M^4 \E\left[\norm{D^2 V_\omega(y_n)}_{\mathrm{op}}^2 \norm{x_n}^2\right] \\
&\leq C_M^4 \E\left[\norm{D^2 V_\omega(y_n)}_{\mathrm{op}}^2\right] R^2 \quad \text{(since $\norm{x_n} \leq R$ a.s.)} \\
&\leq C_M^4 C_2(D) R^2 \quad \text{(by Theorem \ref{thm:moment_bounds})}.
\end{align*}

For the remainder term, $\E[\norm{\tilde{R}_y(\omega, \delta t)}^2] \leq K_y \delta t^6$ with $K_y = \E[\tilde{C}_y(\omega)^2]$.

Thus:
\begin{align*}
\E[\norm{\tau_n^y}^2] &\leq 3\beta_2^2 R^2 \delta t^4 + \frac{\delta t^6}{12} C_M^4 C_2(D) R^2 + 3K_y \delta t^6 \\
&= \left(3\beta_2^2 R^2\right) \delta t^4 + \left(\frac{C_M^4 C_2(D) R^2}{12} + 3K_y\right) \delta t^6.
\end{align*}

For $\delta t \leq 1$, $\delta t^6 \leq \delta t^4$, so there exists $C_y'$ such that:
\[
\E[\norm{\tau_n^y}^2] \leq C_y' \delta t^4.
\]

Similarly for $\tau_n^x$:
\begin{align*}
\norm{\tau_n^x} &\leq 2\abs{\alpha_2} \delta t^2 \norm{x_n} + \norm{\tilde{R}_x(\omega, \delta t)}.
\end{align*}

Squaring and using $(a + b)^2 \leq 2a^2 + 2b^2$:
\begin{align*}
\norm{\tau_n^x}^2 &\leq 6\alpha_2^2 \delta t^4 \norm{x_n}^2 + 3\norm{\tilde{R}_x(\omega, \delta t)}^2.
\end{align*}

Taking expectations:
\begin{align*}
\E[\norm{\tau_n^x}^2] &\leq 6\alpha_2^2 R^2 \delta t^4 + 3K_x \delta t^6 \quad \text{(where $K_x = \E[\tilde{C}_x(\omega)^2]$)} \\
&\leq C_x' \delta t^4 \quad \text{for $\delta t \leq 1$}.
\end{align*}

Finally, since $\norm{\tau_n}^2 = \norm{\tau_n^y}^2 + \norm{\tau_n^x}^2$:
\[
\E[\norm{\tau_n}^2] = \E[\norm{\tau_n^y}^2] + \E[\norm{\tau_n^x}^2] \leq (C_y' + C_x') \delta t^4 =: C_{\mathrm{LTE}} \delta t^4.
\]
\end{proof}

\subsection{Proof of Theorem \ref{thm:global_error}}
\label{app:proof_global_error}

\begin{proof}
Let $e_n = (y(t_n), x(t_n)) - (y_n, x_n)$, where $(y(t), x(t))$ is the exact solution of the original Hamiltonian ODE 
\eqref{eq:original_hamiltonian_y}--\eqref{eq:original_hamiltonian_x}, where $V$ is a Gaussian process satisfying Assumption \ref{ass:gp_regularity}.

From the local truncation error definition:
\[
(y(t_{n+1}), x(t_{n+1})) = \Psi^0_{\delta t}(y(t_n), x(t_n)) + \tau_n,
\]
where $\E[\norm{\tau_n}^2] \leq C_{\mathrm{LTE}} \delta t^4$ by Theorem \ref{thm:ms_lte}.

	The leapfrog scheme ($\alpha=\beta=1$) gives:
\[
(y_{n+1}, x_{n+1}) = \Psi^0_{\delta t}(y_n, x_n).
\]

Thus:
\[
e_{n+1} = \Psi^0_{\delta t}(y(t_n), x(t_n)) - \Psi^0_{\delta t}(y_n, x_n) + \tau_n.
\]

	Now, by Theorem Theorem VIII.3.1 of \cite{hairer2006geometric}, there exists $C>0$, such that 
	\begin{equation}
	\norm {\Psi^0_{\delta t}(y(t_n), x(t_n)) - \Psi^0_{\delta t}(y_n, x_n) }\leq (1+C\delta t^2),
		\label{eq:hairer}
	\end{equation}
from which it follows that:
\[
\norm{e_{n+1}} \leq (1 + C \delta t^2) \norm{e_n} + \norm{\tau_n}.
\]

Square both sides:
\begin{align*}
\norm{e_{n+1}}^2 &\leq (1 + C \delta t^2)^2 \norm{e_n}^2 + 2(1 + C \delta t^2) \norm{e_n} \norm{\tau_n} + \norm{\tau_n}^2.
\end{align*}

Take expectations. Let $u_n = \E[\norm{e_n}^2]$. Then:
\begin{align*}
u_{n+1} &\leq (1 + C \delta t^2)^2 u_n + 2(1 + C \delta t^2) \E[\norm{e_n} \norm{\tau_n}] + \E[\norm{\tau_n}^2].
\end{align*}

Using Cauchy-Schwarz: $\E[\norm{e_n} \norm{\tau_n}] \leq \sqrt{u_n \E[\norm{\tau_n}^2]}$.

By Young's inequality: $2ab \leq \delta t a^2 + \delta t^{-1} b^2$ for any $a, b \geq 0$. Apply with $a = \sqrt{u_n}$ and $b = \sqrt{\E[\norm{\tau_n}^2]}$:
\[
2\sqrt{u_n \E[\norm{\tau_n}^2]} \leq \delta t u_n + \delta t^{-1} \E[\norm{\tau_n}^2].
\]

Thus:
\begin{align*}
u_{n+1} &\leq (1 + C \delta t^2)^2 u_n + (1 + C \delta t^2) \left(\delta t u_n + \delta t^{-1} \E[\norm{\tau_n}^2]\right) + \E[\norm{\tau_n}^2] \\
&= \left[(1 + C \delta t^2)^2 + (1 + C \delta t^2) \delta t\right] u_n + \left[(1 + C \delta t^2) \delta t^{-1} + 1\right] \E[\norm{\tau_n}^2].
\end{align*}

Now $(1 + C \delta t^2)^2 = 1 + 2C \delta t^2 + C^2 \delta t^4 \leq 1 + 3C \delta t^2$ for small $\delta t$.
Also $(1 + C \delta t^2) \delta t \leq 2\delta t$ for small $\delta t$.

Thus:
\[
(1 + C \delta t^2)^2 + (1 + C \delta t^2) \delta t \leq 1 + 3C \delta t^2 + 2\delta t \leq 1 + K \delta t,
\]
where $K = 3C + 2$ (since $\delta t^2 \leq \delta t$ for $\delta t \leq 1$).

For the second term, using $\E[\norm{\tau_n}^2] \leq C_{\mathrm{LTE}} \delta t^4$:
\begin{align*}
\left[(1 + C \delta t^2) \delta t^{-1} + 1\right] \E[\norm{\tau_n}^2] 
&\leq \left[2 \delta t^{-1} + 1\right] C_{\mathrm{LTE}} \delta t^4 \quad \text{(for small $\delta t$)} \\
&= (2\delta t^3 + \delta t^4) C_{\mathrm{LTE}} \\
&\leq 3C_{\mathrm{LTE}} \delta t^3 \quad \text{(for $\delta t \leq 1$)}.
\end{align*}

Thus we have the recurrence:
\[
u_{n+1} \leq (1 + K \delta t) u_n + K' \delta t^3,
\]
where $K' = 3C_{\mathrm{LTE}}$.

Iterate with $u_0 = 0$:
\begin{align*}
u_n &\leq K' \delta t^3 \sum_{j=0}^{n-1} (1 + K \delta t)^j \\
&= K' \delta t^3 \frac{(1 + K \delta t)^n - 1}{K \delta t} \\
&= \frac{K'}{K} \delta t^2 \left[(1 + K \delta t)^n - 1\right].
\end{align*}

Since $n \delta t \leq T$, we have $n \leq T/\delta t$, and:
\[
(1 + K \delta t)^n \leq \left(1 + K \delta t\right)^{T/\delta t} \leq e^{K T}.
\]

Thus:
\[
u_n \leq \frac{K'}{K} \delta t^2 (e^{K T} - 1).
\]

Taking square roots:
\[
\sqrt{u_n} = \sqrt{\E[\norm{e_n}^2]} \leq \sqrt{\frac{K'}{K} (e^{K T} - 1)} \, \delta t =: C_T \delta t.
\]
%

In other words, we obtained:
\[
\sqrt{\E\left[\norm{(y(T), x(T)) - (y_{T/\delta t}, x_{T/\delta t})}^2\right]} \leq C_T \delta t,
\]
where $C_T = \sqrt{\frac{K'}{K} (e^{K T} - 1)}$ with $K = 3C + 2$ and $K' = 3C_{\mathrm{LTE}}$, with $C$ from (\ref{eq:hairer}) 
and $C_{\mathrm{LTE}}$ from Theorem \ref{thm:ms_lte}.

All assumptions are satisfied: Assumptions \ref{ass:gp_regularity}, \ref{ass:sample_paths} ensure the Gaussian process has regular sample paths. Theorem \ref{thm:moment_bounds} provides the necessary moment bounds. Assumption \ref{ass:mass_matrix} gives bounded $M^{-1}$. Assumption \ref{ass:param_expansions} with $\alpha_1 = \beta_1 = 0$ gives the parameter expansions for our modified scheme. Assumption \ref{ass:bounded_iterates} ensures the iterates remain in a compact set.

This completes the proof of mean-square convergence for our modified symplectic scheme introduced in \cite{main}.
\end{proof}

\end{appendix}

\bibliographystyle{natbib}
\bibliography{references4}

\end{document}